\documentclass[a4paper,10pt]{amsart}
\usepackage{amssymb}
\usepackage[all,cmtip]{xy}
\usepackage{amsmath}
\usepackage{chemarrow}
\usepackage{graphicx}
\usepackage{tikz}
\usetikzlibrary{cd,nfold}%
\tikzcdset{diagrams={/tikz/double/.append style=/tikz/nfold}}
\usepackage[colorlinks, 
pdfborder={0 0 1}, 
linkcolor=cyan,
anchorcolor=magenta
citecolor=green
]{hyperref}
\usepackage{mathrsfs}
\usepackage{titletoc}
\usepackage{bm}
\usepackage{bbm}
\usepackage{bbold} 
\usepackage{cite}
\usepackage{subfiles}
\usepackage{extarrows} 
\usepackage[hmarginratio=1:1]{geometry}
\usepackage{colordvi}
\usepackage{color}
\usepackage{stmaryrd}
\usepackage{multirow} 

\allowdisplaybreaks

\DeclareMathOperator{\hh}{H}

\DeclareMathOperator{\kk}{\Bbbk}

\DeclareMathOperator{\N}{\mathbb{N}}

\DeclareMathOperator{\Tor}{Tor}

\DeclareMathOperator{\Z}{\mathbb{Z}}

\numberwithin{equation}{section}

\theoremstyle{definition}

\newtheorem{thm}{Theorem}[section]
\newtheorem{prop}[thm]{Proposition}
\newtheorem{lem}[thm]{Lemma}

\newtheorem{defn}[thm]{Definition}

\newtheorem{ques}[thm]{Question}
\newtheorem{ex}[thm]{Example}

\begin{document}

\title{A note on $N$-Koszul algebras of finite global dimension}

\author{Zeyu Chen}
\address{School of Mathematical Sciences, Zhejiang University, Hangzhou 310058, China}
\email{zeyu98.chen@gmail.com}

\author{Silu Liu}
\address{School of Mathematical Sciences, Fudan University, Shanghai 200433, China}
\email{liusl20@fudan.edu.cn}


\begin{abstract}
    In this note, we prove that, under Kabbaj's weighted polynomial Hilbert-series hypothesis, every $N$-Koszul graded algebra of finite global dimension with $N>2$ is necessarily $3$-Koszul of global dimension $3$.
\end{abstract}

\maketitle

\section{Introduction}

In the development of noncommutative algebraic geometry, Artin and Schelter introduced a class of noncommutative regular algebras\cite{AS87}, which were recognized as Artin--Schelter regular algebras.
 
In the classification of Artin--Schelter regular algebras of global dimension $3$, algebras generated in degree one are divided into two kinds: algebras generated by three variables with three quadratic generating relations, and algebras generated by two variables with two cubic generating relations. The former are Koszul algebras, and the latter are  $3$-Koszul algebras in Berger's terminology. Precisely, Berger introduced $N$-Koszul algebras (see Definition \ref{def N-Koszul}) as a higher-degree analogue of Koszul algebras. Throughout we exclude the trivial case $A=\kk[x]$, which is $N$-Koszul for every $N\geqslant 2$ under Berger's definition.

Despite the extensive literature concerning $N$-Koszul Artin--Schelter regular algebras\cite{DV07,MS16}, for $N>2$, it appears that the only known examples are $3$-Koszul algebras of global dimension $3$, which led Kabbaj to investigate the numerical restrictions for $N$-Koszul algebras \cite{Kab25}. Kabbaj worked under the following hypothesis.

\begin{equation}\label{hypo}\tag{\textbf{H}}  
    \begin{split}
        _A\kk \textit{ has a finite free resolution, and the Hilbert series of $A$ }\\
        \textit{ is of form } \hh_A(t)=\frac{1}{\prod_{i=1}^m (1-t^i)^{n_i}}, \textit{ where } n_i\in \N, n_m>0.
    \end{split} 
\end{equation}

It is worth mentioning that for an Artin--Schelter regular algebra $A$, $_A\kk$ always has a finite free resolution (that is, a free resolution of finite length where each term is finitely generated)\cite{SZ97}. In addition, the Hilbert series of $A$ is conjectured to be of the above form. See \cite[Question 3.3]{Rog24} for detailed information.

\begin{thm}\cite[Theorem 1.2]{Kab25}\label{thm Kabbaj}
    Let $A$ be an $N$-Koszul algebra of global dimension $d$ satisfying hypothesis \eqref{hypo}, where $N>2$. Then
    \begin{enumerate}
        \item $N$ must be a prime number;
        \item $d\equiv 3 \text{ (mod 4) }$;
        \item $d\geqslant \frac{2^N-4}{N}+1$.
    \end{enumerate}
\end{thm}

Following Kabbaj's approach, we show that these restrictions can in fact be pushed to their extreme: under the same hypothesis, no higher-dimensional case can occur.

\begin{thm}\label{thm no-more}
    Let $A$ be an $N$-Koszul algebra of global dimension $d$ satisfying hypothesis \eqref{hypo}, where $N>2$. Then $A$ is $3$-Koszul of global dimension $3$.
\end{thm}

It should be pointed out that a complete answer for the non-existence of more $N$-Koszul Artin--Schelter regular algebras relies on the following open question. 

\begin{ques}\cite[Question 3.3]{Rog24}
    Let $A$ be an Artin--Schelter regular algebra. Is the Hilbert series of $A$ of the form $\hh_A(t)=\frac{1}{\prod\limits_{i=1}^m(1-t^{d_i})}$?
\end{ques}

\section{Preliminaries}
In this section we recall some definitions. Throughout $\kk$ is a fixed field, and modules refer to left modules unless otherwise stated.

A $\Z$-graded $\kk$-algebra $A$ is called \textbf{$\N$-graded} if $A_i=0$ for any $i<0$. An $\N$-graded algebra $A$ is called \textbf{connected graded} if $A_0=\kk$. 

For a graded $A$-module $M$ and $l\in \Z$, $M(l)$ is a graded $A$-module where $M(l)_i:=M_{l+i}$ for any $i\in \Z$. A graded $A$-module $M$ is called \textbf{locally finite} if $\dim_{\kk}(M_i)<\infty$ for any $i\in \Z$. The \textbf{Hilbert series} of a locally finite $A$-module $M$ is defined as $\hh_M(t)=\sum\limits_{i\in \Z}\dim_{\kk}(M_i) t^i$. 

\begin{defn}\cite{Ber01}\label{def N-Koszul}
    Let $A$ be a connected graded locally finite $\kk$-algebra. $A$ is called \textbf{N-Koszul} if $\Tor_i^A(\kk,\kk)$ is concentrated in degree $n(i)$ for any $i\geqslant 0$, where 
    \begin{align*}
        n(i)=
        \begin{cases}
            \frac{i}{2}N, &\text{ if $i$ is even};\\
            \frac{i-1}{2}N+1, &\text{ if $i$ is odd}.
        \end{cases}
    \end{align*}
\end{defn}

\begin{ex}\hspace*{\fill}
    \begin{enumerate}
        \item A connected graded locally finite algebra $A$ is called a \textbf{Koszul algebra} if $\Tor_i^A(\kk,\kk)$ is concentrated in degree $i$ for any $i\geqslant 0$. Koszul algebras are precisely the $2$-Koszul algebras in the above definition. 
        \item \textbf{Graded down-up algebras} are of the form 
        \[A=\kk\langle x,y\rangle/(x^2y-axyx-byx^2, y^2x-ayxy-bxy^2),\]
        where $a,b\in\kk$ with $b\neq 0$. It is well-known that they are $3$-Koszul algebras of global dimension $3$.
    \end{enumerate}
\end{ex}

The following lemma states that for a connected graded $\kk$-algebra $A$, the minimal free resolution of $_A\kk$ determines the Hilbert series of $A$.

\begin{lem}\cite[Lemma 2.3]{SZ97}\label{lem hilbert}
    Let $A$ be a connected graded algebra, and $M$ be a graded $A$-module with an augmented minimal free resolution
    \[ 0\to \bigoplus\limits_{i=1}^{z_d} A(-l_{d,i})\to \cdots \to \bigoplus\limits_{i=1}^{z_1} A(-l_{1,i})\to \bigoplus\limits_{i=1}^{z_0} A(-l_{0,i})\to M\to 0.\]
    Then $\hh_M(t)=c_M(t)\hh_A(t)$, where $c_M(t):=\sum\limits_{j=0}^d(-1)^j(\sum\limits_{i=1}^{z_j}t^{l_{j,i}})$. In particular, if $_A\kk$ has a finite free resolution then $\hh_A(t)=\frac{1}{c_{\kk}(t)}$.
\end{lem}

\section{Proof of Theorem \ref{thm no-more}}

In this section we prove the main theorem. The idea is that, as a result of Lemma \ref{lem hilbert}, for an $N$-Koszul algebra $A$ satisfying Hypothesis \eqref{hypo}, there are severe restrictions on $c_{\kk}(t)$, which is a polynomial with integral coefficients. We prove that there are no such polynomials other than $f(t)=1-2t+2t^3-t^4$, which corresponds to $3$-Koszul algebras of global dimension $3$.

\begin{defn}
    A polynomial $f(t)\in \Z[t]$ is called a \textbf{good polynomial} if:
    \begin{enumerate}
        \item $f(t)=1-\beta_1 t+\beta_2 t^p-\beta_3 t^{p+1}+\beta_4 t^{2p}+\cdots+\beta_{2k} t^{kp}-t^{kp+1}$, where $p$ is an odd prime number, and $\beta_i=\beta_{2k+1-i}>0$, $\forall\ 1\leqslant i\leqslant 2k$;
        \item $f(t)=\prod\limits_{i=1}^{m}(1-t^i)^{n_i}$, where $m\geqslant p-1$, $n_i\in\N$ such that $n_1,\dotsc, n_{p-1}, n_m>0$; $n_{ip}=0$ for $1\leqslant i\leqslant \lfloor m/p\rfloor$.
    \end{enumerate}
\end{defn}

\begin{prop}
    Let $A$ be an $N$-Koszul algebra of global dimension $d$ which satisfies hypothesis \eqref{hypo}. Then $c_{\kk}(t)$ is a good polynomial such that $d=2k+1$. 
\end{prop}

\begin{proof}
    By assumption $_A\kk$ has an augmented minimal free resolution 
    \[ 0\to \bigoplus\limits_{i=1}^{\beta_d} A(-l_{d,i})\to \cdots \to \bigoplus\limits_{i=1}^{\beta_1} A(-l_{1,i})\to A \to \kk \to 0.\]
    Theorem \ref{thm Kabbaj} implies that $N=p$ is an odd prime number, and $d=2k+1$ for some odd number $k$. It follows from the $N$-Koszul condition that $l_{j,i}=n(j)$, $\forall 1\leqslant i\leqslant \beta_j$, $\forall 1\leqslant j\leqslant d$. Therefore 
    \begin{align*}
        c_{\kk}(t)        =&1+\sum\limits_{j=1}^d(-1)^j(\sum\limits_{i=1}^{\beta_j}t^{l_{j,i}})    =1+\sum\limits_{j=1}^d(-1)^j\beta_jt^{n(j)}\\
        =& 1-\beta_1t+\beta_2t^p-\beta_3t^{p+1}+\cdots+\beta_{2k}t^{kp}-\beta_{2k+1}t^{kp+1}.
    \end{align*}
    On the other hand, by hypothesis \eqref{hypo} $c_{\kk}(t)=\prod\limits_{i=1}^{m}(1-t^i)^{n_i}$, which implies that 
    \[c_{\kk}(t)=(-1)^{\sum\limits_{i=1}^m n_i}t^{\sum\limits_{i=1}^m in_i}c_{\kk}(\frac{1}{t})=-t^{kp+1}c_{\kk}(\frac{1}{t}).\]
    By comparing coefficients of both sides we conclude immediately that $\beta_{2k+1}=1$ and $\beta_i=\beta_{2k+1-i}$, $\forall 1\leqslant i\leqslant 2k$. 
    By \cite[Proposition 5.2]{PP05} $\beta_1$ is the number of minimal generators of $A$ (as a $\kk$-algebra). If $\beta_1=1$, then, since $A$ is $N$-homogeneous, $A$ is either $\kk[x]$ or $\kk[x]/(x^N)$. The former is excluded by our standing convention, and the latter has infinite global dimension. Hence $n_1=\beta_1\geqslant 2$, and it follows from \cite[Lemma 3.6]{Kab25} that $n_i>0$, $\forall i=1,\dotsc, p-1$, and $m\geqslant p-1$. The rest follows from \cite[Proposition 3.3]{Kab25}.
\end{proof}

Let $f(t)$ be a good polynomial. Next we prove that
\[f(t)=(1-t)^2(1-t^2)=1-2t+2t^3-t^4.\]

Let $g_0=1$, $g_i=\beta_{2i}$, $\forall 1\leqslant i\leqslant k$. Let $G(x)=\sum\limits_{i=0}^k g_i x^i$, $H(x)=\sum\limits_{i=0}^k g_{k-i}x^{i}$. Then 

\begin{align*}
    f(x)=& G(x^p)-xH(x^p);\\ 
    H(x)=& x^k G(\frac{1}{x}).\\
\end{align*}

Let $S=G(1)=H(1)$, $N_0=\sum\limits_{i=1}^m n_i$.

\begin{prop}
    There exists $a\in \N^*$ such that $S=p^a$, $N_0=a(p-1)+1$.
\end{prop}
 
\begin{proof}
    First note that for a primitive $p$-th root of unity $\omega$,
    \[\prod_{i=1}^{p-1}(1-\omega^i)=p.\]
    Since $n_k=0$ whenever $p\mid k$, and multiplication by $k$ permutes the nonzero residue classes modulo $p$ whenever $p\nmid k$, we have  
    \begin{align*}
        \prod\limits_{i=1}^{p-1} f(\omega^i)
        =\prod\limits_{k=1}^m \prod\limits_{i=1}^{p-1}(1-\omega^{ik})^{n_k}
        =p^{N_0}.
    \end{align*}
    On the other hand,
    \begin{align*}
        f(\omega^i)=G(\omega^{ip})-\omega ^i H(\omega^{ip})=S(1-\omega^i),\forall\ 1\leqslant i\leqslant p-1.
    \end{align*}
    Therefore
    \begin{align*}
        pS^{p-1}=\prod\limits_{i=1}^{p-1}f(\omega^i)=p^{N_0}.
    \end{align*}
    Since $p$ is prime, the rest follows immediately.
\end{proof}

\begin{prop}
    If $a=1$ then $p=3$, $m=2$, $n_1=2$, $n_2=1$.
\end{prop}

\begin{proof}
    Since $f(t)$ is a good polynomial, $0\neq n_2=\binom{n_1}{2}$. So $n_1\geqslant 2$, and
    \[N_0=\sum\limits_{i=1}^m n_i\geqslant \sum\limits_{i=1}^{p-1} n_i\geqslant p.\]
    
    If $a=1$ then $N_0=p$, which implies that 
    \[n_1=2,\ n_i=1,\ \forall\ 2\leqslant i\leqslant p-1;\ n_i=0,\ \forall\ i\geqslant p.\]

    By the requirement of $m$ we conclude that $m=p-1$. If $p=3$ then $(n_1,n_2)=(2,1)$ is as required. If $p>3$ then $f(t)$ being a good polynomial implies that $n_3=n_1n_2-\binom{n_1}{3}=2$, contradicting $n_3=1$. 
\end{proof}

From now on we assume that $a\geqslant 2$, and prove that this is impossible.

First, a direct calculation shows that 
\[G'(1)=\frac{kp+1}{2}p^{a-1},\ H'(1)=\frac{kp-1}{2}p^{a-1}.\]
Consequently $k$ is odd. Let $k=2k'+1$.

\begin{lem}\label{lem new-function}
    Let 
    \[\Phi(x)=e^{-\frac{kp+1}{2}x}f(e^x).\]
    Then $\Phi^{(r)}(0)=0$, $\forall\ r=0,1,\dotsc, a(p-1)$.
\end{lem}

\begin{proof}
    By definition 
    \[\Phi(x)=\sum_{i=0}^kg_i(e^{\lambda_i x}-e^{-\lambda_ix}),\]
    where $\lambda_i=ip-\frac{kp+1}{2}=ip-k'p-\frac{p+1}{2}$, $\forall\ i=0,1,\dotsc,k$.
    It follows directly that 
    \[\text{ord}_{x=0}\Phi(x)= \text{ord}_{x=1}f(x)=N_0=a(p-1)+1.\]
\end{proof}

\begin{lem}
    For $s\in \N$, let 
    \[F(s)=\sum\limits_{b=0}^{k'}g_{k'-b}(\frac{p-1}{2}-(b+1)p)^s+\sum\limits_{b=0}^{k'}g_{k'+b+1}(\frac{p-1}{2}+bp)^s,\] 
    then $F(2r+1)=0,\ \forall\ r=0,1,\dotsc, \frac{a(p-1)}{2}-1$.
\end{lem}

\begin{proof}
    It follows from Lemma \ref{lem new-function} that $\forall\ r=0,1,\dotsc,\frac{a(p-1)}{2}-1$,
    \begin{align*}
        0=& \sum\limits_{i=0}^kg_i\lambda_i^{2r+1}\\
         =& \sum\limits_{i=0}^kg_i(ip-k'p-\frac{p+1}{2})^{2r+1}\\
         =& \sum\limits_{b=0}^{k'}g_{k'-b}(-(b+1)p+\frac{p-1}{2})^{2r+1}+\sum\limits_{b=0}^{k'}g_{k'+b+1}(bp+\frac{p-1}{2})^{2r+1}\\
         = & F(2r+1).
    \end{align*}
\end{proof}

Let 
\[D_n=\sum\limits_{b=0}^{k'}g_{k'-b}(-b-1)^n+\sum\limits_{b=0}^{k'}g_{k'+b+1}b^n,\]
then $D_0=\sum\limits_{i=0}^k g_i=S=p^a$. Moreover
\[F(s)=\sum\limits_{i=0}^s \binom{s}{i}(\frac{p-1}{2})^{s-i}p^iD_i.\]

Let 
\[b_r=\sum\limits_{j=0}^r(-1)^{r-j}\binom{r}{j}(\frac{p-1}{2})^{-(2j+1)}F(2j+1),\]
then $b_r=0$, $\forall\ r=0,1,\dotsc, \frac{a(p-1)}{2}-1$.
 
For a given polynomial $\varphi(x)$, let $\Delta \varphi(x):=\varphi(x)-\varphi(x-1)$ be the first-order difference. For $n\geqslant 2$, higher order differences are given by $\Delta^n \varphi(x):=\Delta(\Delta^{n-1}\varphi(x))$. 

\begin{prop}
    For any $0\leqslant r\leqslant \frac{a(p-1)}{2}-1$, there exist $d_{n,r}\in\Z$ for $r\leqslant n\leqslant 2r+1$, such that 
    \[b_r=\sum\limits_{n=r}^{2r+1}d_{n,r} (\frac{2p}{p-1})^n D_n,\]
    where $d_{r,r}=2^r$.
\end{prop}

\begin{proof}
    By definition
    \begin{align*}
        b_r=& \sum\limits_{j=0}^r(-1)^{r-j}\binom{r}{j}\sum\limits_{n=0}^{2j+1}\binom{2j+1}{n}(\frac{2p}{p-1})^n D_n\\
        =& \sum\limits_{n=0}^{2r+1}(\sum\limits_{j=0}^r (-1)^{r-j}\binom{r}{j}\binom{2j+1}{n})(\frac{2p}{p-1})^nD_n.
    \end{align*}
    Let $d_{n,r}=\sum\limits_{j=0}^r (-1)^{r-j}\binom{r}{j}\binom{2j+1}{n}$. Note that $d_{n,r}=\Delta^r P_n(0)$, where $P_n(x)=\binom{2x+1}{n}$. Consequently $d_{n,r}=0$ whenever $n<r$. Moreover $P_r(x)=\binom{2x+1}{r}$ is a polynomial in $\mathbb{Q}[x]$ of degree $r$ with leading coefficient $\frac{2^r}{r!}$. It follows that
    \[d_{r,r}=\Delta^r P_r(0)=r! \frac{2^r}{r!}=2^r.\]
\end{proof}

Next we prove that $p\mid D_a$. The following divisibility arguments work in $\Z_{(p)}$ instead of $\Z$, where $\Z_{(p)}$ denotes the localization of $\Z$ at the prime ideal $(p)$. 

\begin{lem}
    If $p=3$ then $a\equiv 1(\text{mod }3)$, and consequently $a\geqslant 4$.
\end{lem}

\begin{proof}
    Let 
    \[A=\sum\limits_{i\equiv 1(\text{mod }3)}in_i,\ B=\sum\limits_{i\equiv 2(\text{mod }3)}in_i;\]
    \[N_1=\sum\limits_{i\equiv 1(\text{mod }3)}n_i,\ N_2=\sum\limits_{i\equiv 2(\text{mod }3)}n_i.\]
    Let $\omega$ be a primitive third root of unity. Then the expression
    \[f(x)=\prod\limits_{i=1}^m(1-x^i)^{n_i}\]
    shows that 
    \[\frac{f'(\omega)}{f(\omega)}=-\frac{A}{1-\omega}-\frac{\omega B}{1-\omega^2}=\frac{\omega^2B-A}{1-\omega}.\]
    On the other hand, the expression 
    \[f(x)=G(x^3)-xH(x^3)\]
    implies that 
    \[\frac{f'(\omega)}{f(\omega)}=\frac{(3k+1)(\omega^2-1)}{2(1-\omega)}.\]
    Combining these identities, we obtain that $A=B=\frac{3k+1}{2}$.
    It follows that
    \[N_1\equiv A=B\equiv 2N_2\ (\text{mod }3).\]
    Therefore $2a+1=N_0=N_1+N_2\equiv 3N_2\equiv 0 (\text{mod }3)$, and the rest follows.
\end{proof}

\begin{prop}\label{prop p-adic}
    For any $0\leqslant j\leqslant a-2$, $p^{j+1}\mid D_{a-j}$.
\end{prop}

\begin{proof}
    We prove by induction on $j$.

    First note that for $i\in\{1,2\}$, if $n\equiv i(\text{mod } 2)$ then $x^n\equiv x^i(\text{mod } 3)$. If $p=3$ then either $D_a\equiv D_1(\text{mod } 3)$ or $D_a\equiv D_2(\text{mod } 3)$. Note that $D_0=3^a$, and that
    \[0=F(1)=D_0+3D_1.\]
    Consequently $D_1=-3^{a-1}$. By assumption $a\geqslant 2$, which implies that $3\mid D_1$. Moreover
    \[0=F(3)=D_0+9D_1+27D_2+27D_3,\]
    which implies that $D_2+D_3=2\cdot 3^{a-3}$. But then $D_3\equiv D_1(\text{mod }3)$, which implies that $3\mid D_3$. So $3\mid 2\cdot 3^{a-3}-D_3=D_2$. Combining these identities, we conclude that $3\mid D_a$.

    If $p>3$ then $\frac{a(p-1)}{2}-1\geqslant 2a-1\geqslant a$. Therefore
    \[0=b_a=2^a(\frac{2p}{p-1})^aD_a+\sum\limits_{n=a+1}^{2a+1}d_{n,a} (\frac{2p}{p-1})^n D_n.\]
    It follows that $p^{a+1}\mid 2^a(\frac{2p}{p-1})^a D_a$. Since $\gcd(2,p)=\gcd(p-1,p)=1$, we conclude that $p\mid D_a$.

    Suppose $p^{l+1}\mid D_{a-l}$ for all $0\leqslant l\leqslant j-1$, where $1\leqslant j\leqslant a-2$. By assumption
    \[0=b_{a-j}=\sum\limits_{n=a-j}^{2(a-j)+1}d_{n,a-j} (\frac{2p}{p-1})^n D_n.\]
    If $n\geqslant a+1$ then obviously $p^{a+1}\mid (\frac{2p}{p-1})^n$; if $a-j+1\leqslant n\leqslant a$ then by induction hypothesis $p^{a-n+1}\mid D_{n}$, and consequently $p^{a+1}=p^n\cdot p^{a-n+1}\mid (\frac{2p}{p-1})^n D_n$. It follows that $p^{a+1}\mid 2^{a-j}(\frac{2p}{p-1})^{a-j}D_{a-j}$, which further implies that $p^{j+1}\mid D_{a-j}$.
\end{proof}

\begin{thm}
    There is no good polynomial corresponding to $a\geqslant 2$.
\end{thm}

\begin{proof}
    Since $\frac{a(p-1)}{2}-1\geqslant 1$,
    \[0=F(3)=(\frac{p-1}{2})^3D_0+3(\frac{p-1}{2})^2p D_1+3(\frac{p-1}{2})p^2D_2+p^3D_3.\]
    Recall that $D_0=p^a$ and $D_1=-\frac{p-1}{2}p^{a-1}$. So
    \begin{equation}\label{*}
        3(\frac{p-1}{2})p^2 D_2+ p^3 D_3=2(\frac{p-1}{2})^3p^a.
    \end{equation}
    By Proposition \ref{prop p-adic}, $p^{a-1}\mid D_2$. If $a=2$ then the left hand side of \eqref{*} is divided by $p^3$; if $a\geqslant 3$ then again by Proposition \ref{prop p-adic}, $p^{a-2}\mid D_3$, which implies that the left hand side of \eqref{*} is divided by $p^{a+1}$. Either way we conclude that $p^{a+1}\mid 2(\frac{p-1}{2})^3p^a$, which is a contradiction.
\end{proof}

\section{Acknowledgments}
The authors would like to thank Mengying Hu and Professor Quanshui Wu for helpful discussions. Generative AI tools were used during the exploratory stage of this project to suggest possible approaches and to assist with checking mathematical manipulations, including Gemini 3.1 Deep Think and GPT--5.6 Sol via ChatGPT. All mathematical arguments in the final manuscript were independently verified by the authors, who take full responsibility for the content.

\bibliographystyle{amsalpha}
\bibliography{refs.bib}

@article {Ber01,
    AUTHOR = {Berger, R.},
     TITLE = {Koszulity for nonquadratic algebras},
   JOURNAL = {J. Algebra},
  FJOURNAL = {Journal of Algebra},
    VOLUME = {239},
      YEAR = {2001},
    NUMBER = {2},
     PAGES = {705--734},
      ISSN = {0021-8693,1090-266X},
   MRCLASS = {16S37},
  MRNUMBER = {1832913},
MRREVIEWER = {Dmitriy\ A.\ Rumynin},
       DOI = {10.1006/jabr.2000.8703},
       URL = {https://doi.org/10.1006/jabr.2000.8703},
}

@article {Kab25,
    AUTHOR = {Kabbaj, A.},
     TITLE = {Where are the {$N$}-{K}oszul algebras of finite global
              dimension?},
   JOURNAL = {Proc. Amer. Math. Soc.},
  FJOURNAL = {Proceedings of the American Mathematical Society},
    VOLUME = {153},
      YEAR = {2025},
    NUMBER = {3},
     PAGES = {1013--1024},
      ISSN = {0002-9939,1088-6826},
   MRCLASS = {16W50 (16P90)},
  MRNUMBER = {4862139},
MRREVIEWER = {Peter\ K\'alnai},
       DOI = {10.1090/proc/17103},
       URL = {https://doi.org/10.1090/proc/17103},
}

@article {DV07,
    AUTHOR = {Dubois-Violette, Michel},
     TITLE = {Multilinear forms and graded algebras},
   JOURNAL = {J. Algebra},
  FJOURNAL = {Journal of Algebra},
    VOLUME = {317},
      YEAR = {2007},
    NUMBER = {1},
     PAGES = {198--225},
      ISSN = {0021-8693,1090-266X},
   MRCLASS = {16E10 (16S37 16S38)},
  MRNUMBER = {2360146},
       DOI = {10.1016/j.jalgebra.2007.02.007},
       URL = {https://doi.org/10.1016/j.jalgebra.2007.02.007},
}

@article {MS16,
    AUTHOR = {Mori, I. and Smith, S. P.},
     TITLE = {{$m$}-{K}oszul {A}rtin-{S}chelter regular algebras},
   JOURNAL = {J. Algebra},
  FJOURNAL = {Journal of Algebra},
    VOLUME = {446},
      YEAR = {2016},
     PAGES = {373--399},
      ISSN = {0021-8693,1090-266X},
   MRCLASS = {16E65 (16S38)},
  MRNUMBER = {3421098},
MRREVIEWER = {Liyu\ Liu},
       DOI = {10.1016/j.jalgebra.2015.09.016},
       URL = {https://doi.org/10.1016/j.jalgebra.2015.09.016},
}

@book {PP05,
    AUTHOR = {Polishchuk, A. and Positselski, L.},
     TITLE = {Quadratic algebras},
    SERIES = {University Lecture Series},
    VOLUME = {37},
 PUBLISHER = {American Mathematical Society, Providence, RI},
      YEAR = {2005},
     PAGES = {xii+159},
      ISBN = {0-8218-3834-2},
   MRCLASS = {16S37 (16E30)},
  MRNUMBER = {2177131},
MRREVIEWER = {Ralf\ Fr\"oberg},
       DOI = {10.1090/ulect/037},
       URL = {https://doi.org/10.1090/ulect/037},
}

@article {AS87,
    AUTHOR = {Artin, M. and Schelter, W. F.},
     TITLE = {Graded algebras of global dimension {$3$}},
   JOURNAL = {Adv. in Math.},
  FJOURNAL = {Advances in Mathematics},
    VOLUME = {66},
      YEAR = {1987},
    NUMBER = {2},
     PAGES = {171--216},
      ISSN = {0001-8708},
   MRCLASS = {16A03 (16A48)},
  MRNUMBER = {917738},
MRREVIEWER = {Jean-Pierre\ van Deuren},
       DOI = {10.1016/0001-8708(87)90034-X},
       URL = {https://doi.org/10.1016/0001-8708(87)90034-X},
}

@article {SZ97,
    AUTHOR = {Stephenson, D. R. and Zhang, J. J.},
     TITLE = {Growth of graded {N}oetherian rings},
   JOURNAL = {Proc. Amer. Math. Soc.},
  FJOURNAL = {Proceedings of the American Mathematical Society},
    VOLUME = {125},
      YEAR = {1997},
    NUMBER = {6},
     PAGES = {1593--1605},
      ISSN = {0002-9939,1088-6826},
   MRCLASS = {16P90 (16E10 16P40 16W50)},
  MRNUMBER = {1371143},
MRREVIEWER = {Martha\ K.\ Smith},
       DOI = {10.1090/S0002-9939-97-03752-0},
       URL = {https://doi.org/10.1090/S0002-9939-97-03752-0},
}

@incollection {Rog24,
    AUTHOR = {Rogalski, D.},
     TITLE = {Artin-{S}chelter regular algebras},
 BOOKTITLE = {Recent advances in noncommutative algebra and geometry},
    SERIES = {Contemp. Math.},
    VOLUME = {801},
     PAGES = {195--241},
 PUBLISHER = {Amer. Math. Soc., [Providence], RI},
      YEAR = {[2024] \copyright 2024},
      ISBN = {978-1-4704-7239-9; [9781470476328]},
   MRCLASS = {16E05},
  MRNUMBER = {4756385},
MRREVIEWER = {Padmini\ Veerapen},
       DOI = {10.1090/conm/801/16084},
       URL = {https://doi.org/10.1090/conm/801/16084},
}

\end{document}